\documentclass[12pt,a4paper]{article}

\usepackage{amsmath,amssymb}
\usepackage{amsthm}

\newtheorem{definition}{Definition}
\newtheorem{example}{Example}

\newtheorem{proposition}{Proposition}
\newtheorem{remark}{Remark}
\newtheorem{theorem}{Theorem}

\title{Linear discrete multitime multiple recurrence}

\author{Cristian Ghiu$^1$, Raluca Tulig\u{a}$^2$, Constantin Udri\c ste$^2$}
\date{}

\begin{document}

\maketitle

\begin{center}
{\footnotesize
$^1$University Politehnica of Bucharest,
Faculty of Applied Sciences,

Department of Mathematical Methods and Models,
Splaiul Independentei 313,

Bucharest 060042, Romania;
e-mail: crisghiu@yahoo.com

$^2$University Politehnica of Bucharest,
Faculty of Applied Sciences, Department of Mathematics-Informatics,
Splaiul Independentei 313,
Bucharest 060042, Romania;
e-mails: ralucacoada@yahoo.com; udriste@mathem.pub.ro

}
\end{center}

\begin{abstract}
The multitime multiple recurrences are common in analysis of algorithms,
computational biology, information theory, queueing theory, filters theory,
statistical physics etc.
The theoretical part about them is little or not known. That is why, the aim of our paper
is to formulate and solve problems concerning nonautonomous multitime multiple recurrence equations.
Among other things, we discuss in detail the cases of linear recurrences with constant coefficients,
highlighting in particular the theorems of existence and uniqueness of solutions.
\end{abstract}



\section{Basic statements}

 A multivariate recurrence relation is an equation that recursively defines a multivariate sequence,
 once one or more initial terms are given: each further term of the sequence is defined as a function
 of the preceding terms. Some simply defined recurrence relations can have very complex (chaotic)
 behaviors, and they are a part of the field of mathematics known as nonlinear analysis.
 We can use such recurrences including the Differential Transform Method to
solve completely integrable first order PDEs system with initial conditions.

In this paper we shall refer to
linear discrete multitime multiple recurrence,
giving original results regarding generic properties and existence and uniqueness of solutions.
Also, we seek to provide a fairly thorough and unified exposition of efficient
recurrence relations in both univariate and multivariate settings.
The scientific sources used in this paper are:
filters theory \cite{DM}, \cite{FoM}-\cite{FMa},
\cite{MGS}, \cite{Pr}, \cite{Ro},
general recurrence theory \cite{HK}, \cite{BP}, 
\cite{E}, \cite{W}, \cite{WCM},
our results regarding the diagonal multitime recurrence \cite{GTU}, \cite{GTUT},
and multitime dynamical systems \cite{U1}-\cite{U7}.

Let $m\geq 1$ be an integer number.
We denote ${\bf 1}=(1,1,\ldots, 1) \in \mathbb{Z}^m$. Also, for each
$\alpha \in \{ 1,2, \ldots, m \}$, we denote
$1_{\alpha}=(0,\ldots,0,1,0,\ldots, 0)\in \mathbb{Z}^m$, i.e.,
$1_{\alpha}$ has $1$ on the position $\alpha$ and $0$ otherwise.

On $\mathbb{Z}^m$, we define the relation $``\leq"$:
for $t=(t^1,\ldots, t^m)$,
$s=(s^1,\ldots, s^m)$,
\begin{equation*}
s\leq t\,\,\,\, \mbox{if}\,\,\,
s^{\alpha} \leq t^{\alpha},\,\,
\forall
\alpha \in \{ 1,2, \ldots, m \}.
\end{equation*}
One observes that $``\leq"$ is a partial order relation on $\mathbb{Z}^m$.

Let $M$ be an arbitrary nonvoid set and $t_1\in \mathbb{Z}^m$ be a fixed element.
We consider the functions $F_\alpha \colon \big\{ t \in \mathbb{Z}^m \, \big|\,
t \geq t_1 \big\} \times M \to M$, $\alpha \in \{ 1,2, \ldots, m \}$.

We fix $t_0 \in \mathbb{Z}^m$, $t_0 \geq t_1$.
A first order multitime recurrence of the type
\begin{equation}
\label{rec.alfa}
x(t+1_\alpha) = F_\alpha (t,x(t)),
\quad
\forall t \in \mathbb{Z}^m,\,\,
t\geq t_0,\,\,
\forall \alpha \in \{1,2, \ldots, m\},
\end{equation}
is called a {\it discrete multitime multiple recurrence}.

This model of multiple recurrence can be justified by the fact that a completely integrable first order PDE system
$$\frac{\partial x^i}{\partial t^\alpha}(t)= X^i_\alpha(t,x(t)), \, t \in \mathbb{R}^m$$
can be discretized as
$$x^i(t+1_\alpha)-x^i(t) = X^i_\alpha(t,x(t)),\, t\in \mathbb{Z}^m.$$
The initial (Cauchy) condition, for the PDE system, is translated
into initial condition for the multiple recurrence.

In the paper \cite{GTUd} has shown the following result
\begin{proposition}
\label{alfa.p1}
If for any $(t_0,x_0) \in
\big\{ t \in \mathbb{Z}^m \, \big|\, t \geq t_1 \big\}
\times M$,
there exists at least one solution
$x \colon
\big\{ t \in \mathbb{Z}^m \, \big|\,
t \geq t_0 \big\} \to M$
which verifies the recurrence {\rm (\ref{rec.alfa})} and the initial condition
$x(t_0)=x_0$, then
\begin{equation}
\label{alfap1.1}
F_\alpha(t+1_\beta,F_\beta (t,x))
=
F_\beta(t+1_\alpha,F_\alpha (t,x)),
\quad
\quad
\forall t \geq t_1,\,\,
\forall x \in M,
\end{equation}
$$
\forall \alpha, \beta \in \{1,2, \ldots, m\}.
$$
\end{proposition}

In the same paper \cite{GTUd}, we have proved the following two theorems.
\begin{theorem}
\label{alfa.t3}
Let $M$ be an arbitrary nonvoid set and $t_0 \in \mathbb{Z}^m$.
We consider the functions
$F_\alpha \colon \big\{ t \in \mathbb{Z}^m \, \big|\,
t \geq t_0 \big\} \times M \to M$,
$\alpha \in \{ 1,2, \ldots, m \}$
such that
\begin{equation}
\label{alfat3.1}
F_\alpha(t+1_\beta,F_\beta (t,x))
=
F_\beta(t+1_\alpha,F_\alpha (t,x)),
\end{equation}
$$
\forall t \geq t_0,\,\,
\forall x \in M,
\quad
\forall \alpha, \beta \in \{1,2, \ldots, m\}.
$$
Then, for any $x_0 \in  M$, there exists a unique function
$x \colon \big\{ t \in \mathbb{Z}^m \, \big|\,
t \geq t_0 \big\} \to M$ which verifies
\begin{equation}
\label{alfat3.2}
x(t+1_\alpha) = F_\alpha (t,x(t)),
\quad
\forall t\geq t_0,
\quad
\forall \alpha \in \{1,2, \ldots, m\},
\end{equation}
and the condition $x(t_0)=x_0$.
\end{theorem}

\begin{theorem}
\label{alfa.t4}
Let $M$ be a nonvoid set. For each
$\alpha \in \{ 1,2, \ldots, m \}$, we consider the function
$F_{\alpha} \colon \mathbb{Z}^m \times M \to M$,
to which we associate the recurrence equation
\begin{equation}
\label{ecalfat4.1}
x(t+1_\alpha) = F_\alpha (t,x(t)),
\quad
\forall \alpha \in \{1,2, \ldots, m\}.
\end{equation}

The following statements are equivalent:

\noindent
$i)$ For any $\alpha\in \{1,2, \ldots, m\}$
and any $t\in \mathbb{Z}^m$, the functions $F_{\alpha}(t, \cdot)$ are bijective
and
\begin{equation}
\label{ecalfat4.2}
F_\alpha(t+1_\beta,F_\beta (t,x))
=
F_\beta(t+1_\alpha,F_\alpha (t,x)),
\end{equation}
$$
\forall (t,x) \in \mathbb{Z}^m \times M,
\,\,\,
\forall \alpha, \beta \in \{1,2, \ldots, m\}.
$$

\noindent
$ii)$ For any pair $(t_0,x_0) \in \mathbb{Z}^m \times M$,
and any index $\alpha_0 \in \{ 1,2, \ldots, m \}$,
there exists a unique function
$x \colon \big\{ t \in \mathbb{Z}^m \, \big|\,
t \geq t_0-1_{\alpha_0} \big\} \to M$,
which, for each $t \geq t_0-1_{\alpha_0}$ verifies the relations
{\rm (\ref{ecalfat4.1})},
and also the condition $x(t_0)=x_0$.

\vspace{0.1 cm}
\noindent
$iii)$ For any $t_0, t_1 \in \mathbb{Z}^m$, with $t_1 \leq t_0$,
and any $x_0 \in  M$, there exists a function
$x \colon \big\{ t \in \mathbb{Z}^m \, \big|\,
t \geq t_1 \big\} \to M$,
which, for any $t \geq t_1$, verifies the relations
{\rm (\ref{ecalfat4.1})}, and also the condition $x(t_0)=x_0$.

\vspace{0.1 cm}
\noindent
$iv)$ For any
$(t_0,x_0) \in \mathbb{Z}^m \times M$,
there exists a unique function
$x \colon \mathbb{Z}^m \to M$,
which, for any $t \in \mathbb{Z}^m $, verifies the relation
{\rm (\ref{ecalfat4.1})},
and also the condition $x(t_0)=x_0$.
\end{theorem}

\section{Linear discrete multitime \\multiple recurrence}

Let $K$ be a field.
We denote by ${\cal Z}$ one of the sets
$\mathbb{Z}^m$ or
$ \big\{ t \in \mathbb{Z}^m \, \big|\,
t \geq t_1 \big\}$  (with $t_1 \in \mathbb{Z}^m$).

For each $\alpha \in \{1,2, \ldots, m\}$,
we consider the functions
$$
A_\alpha \colon \mathcal{Z} \to \mathcal{M}_{n}(K),
\quad
b_\alpha \colon \mathcal{Z} \to K^n=\mathcal{M}_{n,1}(K),
$$
which define the recurrence
\begin{equation}
\label{rec.alfa.lin}
x(t+1_\alpha) = A_\alpha(t)x(t)+b_\alpha(t),
\quad
\forall \alpha \in \{1,2, \ldots, m\},
\end{equation}
with the unknown function
$x \colon \big\{ t \in \mathcal{Z} \, \big|\,
t \geq t_0 \big\} \to K^n=\mathcal{M}_{n,1}(K)$,
$t_0 \in \mathcal{Z}$.
This is a particular case of
discrete multitime multiple recurrence (\ref{rec.alfa}),
with $M=K^n$ and
$F_{\alpha}(t,x)=A_{\alpha}(t)x+b_\alpha(t)$.

\begin{theorem}
\label{alfa.t5}
$a)$ If, for any $(t_0,x_0) \in \mathcal{Z} \times K^n$,
there exists at least one function
$x \colon \big\{ t \in \mathcal{Z} \, \big|\,
t \geq t_0 \big\} \to K^n$, which,
for any $t\geq t_0$, verifies the recurrence
{\rm (\ref{rec.alfa.lin})} and the condition $x(t_0)=x_0$, then
\begin{equation}
\label{ecalfat5.1}
A_\alpha(t+1_\beta)A_\beta(t)
=
A_\beta(t+1_\alpha)A_\alpha(t),
\end{equation}
\begin{equation}
\label{ecalfat5.2}
A_\alpha(t+1_\beta)b_\beta(t)+b_\alpha(t+1_\beta)
=
A_\beta(t+1_\alpha)b_\alpha(t)+b_\beta(t+1_\alpha),
\end{equation}
$$
\forall t \in  \mathcal{Z},
\quad
\forall \alpha, \beta \in \{1,2, \ldots, m\}.
$$

$b)$ If the relations {\rm (\ref{ecalfat5.1})},
{\rm (\ref{ecalfat5.2})} are satisfied,
then, for any $(t_0,x_0) \in \mathcal{Z} \times K^n$,
there exists a unique function $x \colon \big\{ t \in \mathcal{Z} \, \big|\,
t \geq t_0 \big\} \to K^n$, which,
for any $t\geq t_0$ verifies the recurrence
{\rm (\ref{rec.alfa.lin})} and the initial condition $x(t_0)=x_0$.
\end{theorem}

\begin{proof}
$a)$ One applies the Proposition \ref{alfa.p1}.
The relations (\ref{alfap1.1}) become
\begin{equation*}
A_\alpha(t+1_\beta)
\big(A_\beta(t)x+b_\beta(t)\big)
+b_\alpha(t+1_\beta)
=
A_\beta(t+1_\alpha)
\big(A_\alpha(t)x+b_\alpha(t)\big)
+b_\beta(t+1_\alpha),
\end{equation*}
\begin{equation*}
\forall x \in K^n,\,\,
\forall t \geq t_1.
\end{equation*}

In the case ${\cal Z}=\mathbb{Z}^n$, the point $t_1$ is arbitrary.
We deduce that the foregoing relations are true $\forall x \in K^n$, $\forall t \in {\cal Z}$.
Setting $x=0$, we obtain the relations
(\ref{ecalfat5.2}). It follows that
\begin{equation}
\label{ecalfat5.3}
A_\alpha(t+1_\beta)A_\beta(t)x
=
A_\beta(t+1_\alpha)A_\alpha(t)x,
\quad
\forall x \in K^n,\,\,
\forall t \in \mathcal{Z}.
\end{equation}

For $j \in \{1,2, \ldots, n\}$,
let $e_j=(0,\ldots,0,1,0,\ldots, 0)^{\top}$
be the column of $K^n$ which has $1$ on the position $j$
and $0$ in rest. From (\ref{ecalfat5.3}) it follows
\begin{equation*}
A_\alpha(t+1_\beta)A_\beta(t)
\cdot
\Big(
  \begin{array}{cccc}
    e_1 & e_2 & ... & e_n \\
  \end{array}
\Big)
=
A_\beta(t+1_\alpha)A_\alpha(t)
\cdot
\Big(
  \begin{array}{cccc}
    e_1 & e_2 & ... & e_n \\
  \end{array}
\Big),
\end{equation*}
equivalent to
$$A_\alpha(t+1_\beta)A_\beta(t)I_n=
A_\beta(t+1_\alpha)A_\alpha(t)I_n,$$
i.e., the relations (\ref{ecalfat5.1}).
\end{proof}

For $F_{\alpha}(t,x)=A_{\alpha}(t)x+b_\alpha(t)$,
one observes that the function $F_{\alpha}(t,\cdot)$ is injective
if and only if $F_{\alpha}(t,\cdot)$ is surjective, i.e.,
if and only if the matrix $A_{\alpha}(t)$ is invertible.

Hence, in this case, the Theorem \ref{alfa.t4} can be written
\begin{theorem}
\label{alfa.t6}
For each $\alpha \in \{1,2, \ldots, m\}$, we consider the functions
$$
A_\alpha \colon \mathbb{Z}^m \to \mathcal{M}_{n}(K),\
\quad
b_\alpha \colon \mathbb{Z}^m \to K^n=\mathcal{M}_{n,1}(K).
$$
which define the recurrence {\rm (\ref{rec.alfa.lin})}.

The following statements are equivalent:

\vspace{0.1 cm}
\noindent
$i)$ For any
$\alpha\in \{1,2, \ldots, m\}$
and any $t\in \mathbb{Z}^m$,
the matrix $A_{\alpha}(t)$ is invertible
and $\forall t \in \mathbb{Z}^m$,
$\forall \alpha, \beta \in \{1,2, \ldots, m\}$
the relations {\rm (\ref{ecalfat5.1})}, {\rm (\ref{ecalfat5.2})}
hold.

\vspace{0.1 cm}
\noindent
$ii)$ For any pair $(t_0,x_0) \in \mathbb{Z}^m \times K^n$,
and any $\alpha_0 \in \{ 1,2, \ldots, m \}$,
there exists at least one function
$x \colon \big\{ t \in \mathbb{Z}^m \, \big|\,
t \geq t_0-1_{\alpha_0} \big\} \to K^n$,
which, for any $t \geq t_0-1_{\alpha_0}$, verifies the relations
{\rm (\ref{rec.alfa.lin})},
and also the condition $x(t_0)=x_0$.

\vspace{0.1 cm}
\noindent
$iii)$ For any pair $(t_0,x_0) \in \mathbb{Z}^m \times K^n$,
and any $\alpha_0 \in \{ 1,2, \ldots, m \}$,
there exists a unique function
$x \colon \big\{ t \in \mathbb{Z}^m \, \big|\,
t \geq t_0-1_{\alpha_0} \big\} \to K^n$,
which, for any
$t \geq t_0-1_{\alpha_0}$, verifies the relations
{\rm (\ref{rec.alfa.lin})},
and also the condition $x(t_0)=x_0$.

\vspace{0.1 cm}
\noindent
$iv)$ For any $t_0, t_1 \in \mathbb{Z}^m$,
with $t_1 \leq t_0$,
and for any $x_0 \in  K^n$,
there exists a unique function
$x \colon \big\{ t \in \mathbb{Z}^m \, \big|\,
t \geq t_1 \big\} \to K^n$,
which, for any $t \geq t_1$, verifies the relations
{\rm (\ref{rec.alfa.lin})},
and also the condition $x(t_0)=x_0$.

\vspace{0.1 cm}
\noindent
$v)$ For any pair
$(t_0,x_0) \in \mathbb{Z}^m \times K^n$,
there exists a unique function
$x \colon \mathbb{Z}^m \to K^n$,
which, for any $t \in \mathbb{Z}^m $, verifies the relations
{\rm (\ref{rec.alfa.lin})},
and also the condition $x(t_0)=x_0$.
\end{theorem}

\begin{proof}
The equivalence of the statements
$i)$, $iii)$, $iv)$, $v)$ follows from the Theorem \ref{alfa.t4}.
Since the implication $iii) \Longrightarrow ii)$ is obvious,
it is sufficient to prove the implication $ii) \Longrightarrow i)$.

$ii)\Longrightarrow i)$:
The relations {\rm (\ref{ecalfat5.1})}, {\rm (\ref{ecalfat5.2})}
follow from the Theorem \ref{alfa.t5}.

Let $F_{\alpha} \colon \mathbb{Z}^m \times K^n \to K^n$,
$F_{\alpha}(t,x)=A_{\alpha}(t)x+b_\alpha(t)$,
$\forall (t,x) \in \mathbb{Z}^m \times K^n$.

Let $\alpha_0 \in \{ 1,2, \ldots, m \}$ and $t_0 \in \mathbb{Z}^m$.

Let $y \in K^n$. There exists a function $x(\cdot)$ which verifies
{\rm (\ref{rec.alfa.lin})}, $\forall t \geq t_0-1_{\alpha_0}$
and the condition $x(t_0)=y$.

For $t = t_0-1_{\alpha_0}$, one obtains
$y=F_{\alpha_0}(t_0-1_{\alpha_0}, x(t_0-1_{\alpha_0}))$.
Since $y$ is arbitrary, it follows that
$F_{\alpha_0} (t_0-1_{\alpha_0},\cdot)$ is surjective,
which is equivalent to that the matrix
$A_{\alpha_0}(t_0-1_{\alpha_0})$ is invertible.
Since $t_0$ is arbitrary, it follows that, for any $t \in \mathbb{Z}^m$,
the matrices $A_{\alpha_0}(t)$ are invertible;
here also $\alpha_0 \in \{ 1,2, \ldots, m \}$
is arbitrary.
\end{proof}

\begin{remark}
\label{compatconst}
If the functions $A_{\alpha}(\cdot)$ and $b_{\alpha}(\cdot)$ are constants,
then the relations
{\rm (\ref{ecalfat5.1})}, {\rm (\ref{ecalfat5.2})}
become
\begin{equation}
\label{ecalfat5.1cst}
A_\alpha A_\beta
=
A_\beta A_\alpha
\end{equation}
\begin{equation}
\label{ecalfat5.2cst}
(A_\alpha-I_n)b_\beta
=
(A_\beta-I_n)b_\alpha.
\end{equation}
\end{remark}

\section{Fundamental (transition) matrix}

We denote by $\mathcal{Z}$ one of the sets
$\mathbb{Z}^m$ or
$ \big\{ t \in \mathbb{Z}^m \, \big|\,
t \geq t_1 \big\}$  (with $t_1 \in \mathbb{Z}^m$).

Consider the functions
$A_\alpha \colon \mathcal{Z} \to \mathcal{M}_{n}(K)$,
$\alpha \in \{1,2, \ldots, m\}$,
which define the linear homogeneous recurrence
\begin{equation}
\label{rec.alfa.omog}
x(t+1_\alpha) = A_\alpha(t)x(t),
\quad
\forall \alpha \in \{1,2, \ldots, m\},
\end{equation}
with the unknown function
$x \colon \big\{ t \in \mathcal{Z} \, \big|\,
t \geq t_0 \big\} \to K^n=\mathcal{M}_{n,1}(K)$,
$t_0 \in \mathcal{Z}$.

\begin{proposition}
\label{alfa.o3}
Suppose that the relations {\rm (\ref{ecalfat5.1})} hold true.

For each $t_0 \in \mathcal{Z}$ and $X_0 \in \mathcal{M}_n(K)$
there exists a unique matrix solution
$$X \colon
\big\{ t \in \mathcal{Z} \, \big|\, t \geq t_0 \big\}
\to \mathcal{M}_n(K)$$
of the recurrence
\begin{equation}
\label{ecalfao3.1}
X(t+1_\alpha) = A_\alpha(t)X(t),
\quad
\forall \alpha \in \{1,2, \ldots, m\},
\end{equation}
with the condition $X(t_0)=X_0$.
\end{proposition}
\begin{proof}
For the $n$ recurrences which are equivalent to the matrix recurrence,
we apply the Theorem {\rm \ref{alfa.t5}}.
\end{proof}

For each $t_0 \in \mathcal{Z}$, we denote
$$\chi(\,\cdot\, , t_0) \colon
\big\{ t \in \mathcal{Z} \, \big|\, t \geq t_0 \big\}
\to \mathcal{M}_n(K),$$
the unique matrix solution of the recurrence {\rm (\ref{ecalfao3.1})}
which verifies $X(t_0)=I_n$.

\begin{definition}
\label{alfa.d1}
Suppose that the relations {\rm (\ref{ecalfat5.1})} hold true.

The matrix function
$$
\chi(\,\cdot\, , \cdot\,)
\colon \big\{ (t,s) \in \mathcal{Z} \times \mathcal{Z}\, \big|\,
t \geq s \big\} \to \mathcal{M}_n(K)
$$
is called fundamental (transition) matrix associated to the linear homogeneous recurrence
{\rm (\ref{rec.alfa.omog})}.
\end{definition}

For $\alpha \in \{1,2, \ldots, m\}$
and $k \in \mathbb{N}$, we define the function
$$
C_{\alpha,\, k} \colon \mathcal{Z} \to \mathcal{M}_{n}(K),
$$
\begin{equation}
\label{Cdef}
C_{\alpha,\, k}(t)=
\left
\{\begin{array}{ccc}
\displaystyle \prod_{j=1}^k A_\alpha(t+(k-j)\cdot 1_\alpha) & \hbox{if} &k \geq 1\\
I_n & \hbox{if}& k=0.
\end{array}
\right.
\end{equation}

\begin{proposition}
\label{alfa.p7}
Suppose that the relations {\rm (\ref{ecalfat5.1})} hold true.

The matrix functions $\chi(\cdot)$ and $C_{\alpha,k}(\cdot)$ have the properties:

$a)$ $\chi(t,s)\chi(s,r)=\chi(t,r)$,\,\,
$\forall t,s,r \in  \mathcal{Z}$, with $t\geq s\geq r$.

$b)$ $\chi(s,s)=I_n$,\,\, $\forall s \in \mathcal{Z}$.

$c)$ $\chi(t+k \cdot 1_\alpha , s)=
C_{\alpha,\, k}(t) \cdot \chi(t,s)$,\,\,
$\forall k \in \mathbb{N}$,
$\forall t,s \in  \mathcal{Z}$, with $t\geq s$.

$d)$ $C_{\alpha,\, k}(t)=\chi(t+k \cdot 1_\alpha , t)$,\,\,
$\forall k \in \mathbb{N}$, $\forall t \in  \mathcal{Z}$.

$e)$ $C_{\alpha,\, k}(t+p \cdot 1_{\beta})
C_{\beta,\, p}(t)
=
C_{\beta,\, p}(t+k \cdot 1_\alpha)
C_{\alpha,\, k}(t)$,\,\,
$\forall k,p \in \mathbb{N}$, $\forall t \in \mathcal{Z}$.

$f)$ For any $t, s \in \mathcal{Z}$ with $t \geq s$, we have
\begin{equation*}
\chi(t,s)=
\prod_{\alpha =1}^{m-1}
C_{\alpha, t^\alpha - s^\alpha}(s^1,...,s^\alpha, t^{\alpha +1},...,t^m)
\cdot
C_{m,\, t^m-s^m}(s^1,s^2,\ldots, s^{m-1},s^m).
\end{equation*}

$g)$ For any $t, s \in \mathcal{Z}$ with $t \geq s$, the fundamental matrix
$\chi(t,s)$ is invertible if and only if, for any $\alpha \in \{1,2, \ldots, m\}$
and any $t\in \mathcal{Z}$, the matrix $A_{\alpha}(t)$ is invertible.

$h)$ For any $\alpha \in \{1,2, \ldots, m\}$,
any $k\in \mathbb{N}$ and for any $t\in \mathcal{Z}$,
$C_{\alpha,\, k}(t)$ is invertible if and only if, for any $\alpha \in \{1,2, \ldots, m\}$
and any $t\in \mathcal{Z}$, the matrix $A_{\alpha}(t)$ is invertible.

$i)$ If $\forall \alpha \in \{1,2, \ldots, m\}$,
$\forall t\in \mathcal{Z}$, the matrix $A_{\alpha}(t)$ is invertible, then
$\forall t,s,t_0 \in  \mathcal{Z}$, with $t\geq s\geq t_0$, we have
$\chi(t,s)=\chi(t,t_0)\chi(s,t_0)^{-1}$.

$j)$ If $\forall \alpha \in \{1,2, \ldots, m\}$,
the matrix functions $A_{\alpha}(\cdot)$ are constant, then
\begin{equation*}
C_{\alpha,\, k}(t)=A_{\alpha}^k,
\quad
\forall k\in \mathbb{N},\,
\forall t\in \mathbb{Z}^m,\,
\forall \alpha \in \{1,2, \ldots, m\},
\end{equation*}
\begin{equation*}
\chi(t,s)=
A_1^{(t^1-s^1)}A_2^{(t^2-s^2)}
\cdot
\ldots
\cdot
A_m^{(t^m-s^m)},
\quad
\forall
t, s \in \mathbb{Z}^m,\,\,
\mbox{with}\,\, t \geq s.
\end{equation*}
\end{proposition}

\begin{proof}
$b)$ It follows directly from the definition of the function $\chi(\cdot , \cdot)$.

$a)$ We fix $s,r$, with $s\geq r$.
Let $Y_1,Y_2  \colon
\big\{ t \in \mathcal{Z} \, \big|\, t \geq s \big\} \to \mathcal{M}_n(K)$,
$$
Y_1(t)=\chi(t,s)\chi(s,r),
\quad
Y_2(t)=\chi(t,r),
\quad
\forall t \geq s.
$$
Then
$$
Y_1(t+1_{\alpha})=\chi(t+1_{\alpha},s)\chi(s,r)
=
A_{\alpha}(t)\chi(t,s)\chi(s,r)
=A_{\alpha}(t)Y_1(t);
$$
$$
Y_1(s)=\chi(s,s)\chi(s,r)=
I_n\chi(s,r)=\chi(s,r)=Y_2(s).
$$
It follows that the functions
$Y_1(\cdot)$ and $Y_2(\cdot)$
are both solutions of the recurrence
{\rm (\ref{ecalfao3.1})}
and coincide for $t=s$. From uniqueness property,
it follows that $Y_1(\cdot)$ and $Y_2(\cdot)$
coincide; hence $\chi(t,s)\chi(s,r)=\chi(t,r)$,
$\forall t\geq s$.

$c)$ Induction after $k$:

For $k=0$, the statement is obvious.

For $k=1$: the equality
$\chi(t+ 1_\alpha , s)=C_{\alpha,\, 1}(t)\cdot \chi(t,s)$
is equivalent to
$\chi(t+ 1_\alpha , s)=A_{\alpha}(t)\chi(t,s)$,
that is true, according the definition of $\chi(\cdot, s)$.

Let $k \geq 2$. Suppose the statement is true for $k-1$ and we shall prove for $k$.
\begin{equation*}
\chi(t+k \cdot 1_\alpha , s)
=
A_{\alpha}(t+(k-1)\cdot 1_\alpha)
\chi(t+(k-1) \cdot 1_\alpha , s)
\end{equation*}
\begin{equation*}
=
A_{\alpha}(t+(k-1)\cdot 1_\alpha)
C_{\alpha,\, k-1}(t) \cdot \chi(t,s)
\end{equation*}
\begin{equation*}
=
A_{\alpha}(t+(k-1)\cdot 1_\alpha)
\cdot
A_{\alpha}(t+(k-2)\cdot 1_\alpha)
\cdot \ldots \cdot A_{\alpha}(t+ 1_\alpha)A_{\alpha}(t)
\cdot \chi(t,s)
\end{equation*}
\begin{equation*}
=C_{\alpha,\, k}(t)\chi(t,s).
\end{equation*}

$d)$ In the equality from the step $c)$, we set $s=t$.

We obtain
$\chi(t+k \cdot 1_\alpha , t)=C_{\alpha,\, k}(t)\chi(t,t)
=C_{\alpha,\, k}(t)$.

$e)$ We use the step $d)$.
$$
C_{\alpha,\, k}(t+p \cdot 1_{\beta})
C_{\beta,\, p}(t)
=
$$
$$
=
\chi(t+p \cdot 1_{\beta}+k \cdot 1_\alpha , t+p \cdot 1_{\beta})
\chi(t+p \cdot 1_{\beta} , t)
=
\chi(t+p \cdot 1_{\beta}+k \cdot 1_\alpha , t).
$$
Analogously, one shows that
$C_{\beta,\, p}(t+k \cdot 1_\alpha)C_{\alpha,\, k}(t)
=
\chi(t+k \cdot 1_\alpha+p \cdot 1_{\beta} , t)$.

$f)$ One uses the step $c)$.
\begin{equation*}
\chi(t,s)=\chi(t-(t^1-s^1)\cdot 1_{1}+(t^1-s^1) \cdot 1_{1} , s)
=
\end{equation*}
\begin{equation*}
=
C_{1,\, t^1-s^1}(t-(t^1-s^1)\cdot 1_{1}) \cdot
\chi(t-(t^1-s^1)\cdot 1_{1},s)=
\end{equation*}
\begin{equation*}
=
C_{1,\, t^1-s^1}(s^1,t^2,\ldots, t^m)
\chi \big((s^1,t^2,\ldots, t^m),s \big)=
\end{equation*}
\begin{equation*}
=
C_{1,\, t^1-s^1}(s^1,t^2,\ldots, t^m)
\chi \big(
(s^1,t^2,\ldots, t^m)-(t^2-s^2)\cdot 1_{2}+(t^2-s^2) \cdot 1_{2},s
\big)=
\end{equation*}
\begin{equation*}
=
C_{1,\, t^1-s^1}(s^1,t^2,\ldots, t^m)
C_{2,\, t^2-s^2}
\big(
(s^1,t^2,\ldots, t^m)-(t^2-s^2)\cdot 1_{2}
\big)
\cdot
\end{equation*}
\begin{equation*}
\cdot
\chi
\big(
(s^1,t^2,\ldots, t^m)-(t^2-s^2)\cdot 1_{2} ,s
\big)=
\end{equation*}
\begin{equation*}
=
C_{1,\, t^1-s^1}(s^1,t^2,\ldots, t^m)
\cdot
C_{2,\, t^2-s^2}(s^1,s^2,t^3, \ldots, t^m)
\cdot
\chi
\big((s^1,s^2,t^3 \ldots, t^m),s \big)
=
\end{equation*}
\begin{equation*}
=\ldots =
C_{1,\, t^1-s^1}(s^1,t^2,\ldots, t^m)
\cdot
C_{2,\, t^2-s^2}(s^1,s^2,t^3,\ldots, t^m) \cdot
\end{equation*}
\begin{equation*}
\cdot
\ldots
\cdot
C_{m-1,\, t^{m-1}-s^{m-1}}(s^1,s^2,\ldots, s^{m-1},t^m)
\cdot
\chi
\big((s^1,s^2,\ldots, s^{m-1},t^m),s \big)=
\end{equation*}
\begin{equation*}
=
C_{1,\, t^1-s^1}(s^1,t^2,\ldots, t^m)
\cdot
C_{2,\, t^2-s^2}(s^1,s^2,t^3,\ldots, t^m)
\cdot
\end{equation*}
\begin{equation*}
\cdot
\ldots
\cdot
C_{m-1,\, t^{m-1}-s^{m-1}}(s^1,s^2,\ldots, s^{m-1},t^m)
\cdot
\end{equation*}
\begin{equation*}
\cdot
\chi
\big((s^1,s^2,\ldots, s^{m-1},t^m)-(t^m-s^m)\cdot 1_{m}
+(t^m-s^m)\cdot 1_{m},s \big)=
\end{equation*}
\begin{equation*}
=
C_{1,\, t^1-s^1}(s^1,t^2,\ldots, t^m)
\cdot
C_{2,\, t^2-s^2}(s^1,s^2,t^3,\ldots, t^m)
\cdot
\end{equation*}
\begin{equation*}
\cdot
\ldots
\cdot
C_{m-1,\, t^{m-1}-s^{m-1}}(s^1,s^2,\ldots, s^{m-1},t^m)
\cdot
\end{equation*}
\begin{equation*}
\cdot
C_{m,\, t^m-s^m}
\big(
(s^1,s^2,\ldots, s^{m-1},t^m)-(t^m-s^m)\cdot 1_{m}
\big)
\cdot
\end{equation*}
\begin{equation*}
\cdot
\chi
\big((s^1,s^2,\ldots, s^{m-1},t^m)-(t^m-s^m)\cdot 1_{m},s
\big)=
\end{equation*}
\begin{equation*}
=
C_{1,\, t^1-s^1}(s^1,t^2,\ldots, t^m)
\cdot
C_{2,\, t^2-s^2}(s^1,s^2,t^3,\ldots, t^m)
\cdot
\end{equation*}
\begin{equation*}
\cdot
\ldots
\cdot
C_{m-1,\, t^{m-1}-s^{m-1}}(s^1,s^2,\ldots, s^{m-1},t^m)
\cdot
C_{m,\, t^m-s^m}(s^1,s^2,\ldots, s^{m-1},s^m)
\cdot
\end{equation*}
\begin{equation*}
\cdot
\chi
\big((s^1,s^2,\ldots, s^{m-1},s^m),s
\big);
\end{equation*}
we obtain the desired identity since $\chi(s,s)=I_n$.

$g)$ and $h)$ If all the matrices
$\chi(\cdot, \cdot)$ are invertible, then from the equality
$\chi(t+1_\alpha,s)= A_\alpha(t)\chi(t,s)$
it follows that $A_\alpha(t)$ is invertible.

If all matrices $A_\alpha(t)$ are invertible, then $C_{\alpha,\, k}(t)$ is invertible
since $C_{\alpha,\, k}(t)$ is either $I_n$,
or a product of the matrices $A_\alpha(\cdot)$.

If all the matrices $C_{\alpha,\, k}(t)$
are invertible, then $\chi(t, s)$
is invertible since $\chi(t, s)$ is a product of matrices $C_{\alpha,\, k}(\cdot)$,
according the step $f)$.

$i)$ The matrix $\chi(s,t_0)$ is invertible.

From the relation $\chi(t,s)\chi(s,t_0)=\chi(t,t_0)$,
we obtain
$$\chi(t,s)=\chi(t,t_0)\chi(s,t_0)^{-1}.$$

$j)$ The relation $C_{\alpha,\, k}(t)=A_{\alpha}^k$
follows directly from the definition of $C_{\alpha,\, k}(t)$.
The second equality required
is obtained using the step $f)$.
\end{proof}

\begin{example}
\label{alfa.exmp3}
The formula at point $f)$ express
explicitly the fundamental matrix $\chi(t,s)$
in function of the matrices $A_{\alpha}(\cdot)$.

For example, for $m=2$, $t=(t^1,t^2)$, $s=(s^1,s^2)$,
with $t^1 \geq s^1+1$, $t^2 \geq s^2+1$, one obtains
\begin{equation*}
\chi(t,s)=
C_{1,\, t^1-s^1}(s^1,t^2)\cdot C_{2,\, t^2-s^2}(s^1,s^2)
=
\end{equation*}
\begin{equation*}
=
A_{1}(t^1-1,t^2)A_{1}(t^1-2,t^2)
\cdot \ldots \cdot A_{1}(s^1+1,t^2)A_{1}(s^1,t^2)
\cdot
\end{equation*}
\begin{equation*}
\cdot
A_{2}(s^1,t^2-1)A_{2}(s^1,t^2-2)
\cdot \ldots \cdot A_{2}(s^1,s^2+1)A_{2}(s^1,s^2).
\end{equation*}
\end{example}

The following result can be proved easily by direct computation.
\begin{proposition}
\label{alfa.p8}
We denote by $\mathcal{Z}$ one of the sets
$\mathbb{Z}^m$ or
$ \big\{ t \in \mathbb{Z}^m \, \big|\,
t \geq t_1 \big\}$  (with $t_1 \in \mathbb{Z}^m$).

We consider the functions $A_\alpha \colon \mathcal{Z} \to \mathcal{M}_{n}(K)$,
$\alpha \in \{1,2, \ldots, m\}$, for which the relations {\rm (\ref{ecalfat5.1})} are
satisfied. Let $(t_0,x_0) \in \mathcal{Z} \times K^n$.
Then, the unique function
$x \colon \big\{ t \in \mathcal{Z} \, \big|\,
t \geq t_0 \big\} \to K^n=\mathcal{M}_{n,1}(K)$,
which, for any $t \geq t_0$, verifies the recurrence {\rm (\ref{rec.alfa.omog})}
and the initial condition $x(t_0)=x_0$, is
\begin{equation*}
x(t)=\chi(t,t_0)x_0,
\quad
\forall t \geq t_0.
\end{equation*}

If $\forall \alpha \in \{1,2, \ldots, m\}$,
the matrix functions $A_{\alpha}(\cdot)$ are constant,
then
\begin{equation}
\label{solautconst}
x(t)=
A_1^{(t^1-t_0^1)}A_2^{(t^2-t_0^2)}
\cdot
\ldots
\cdot
A_m^{(t^m-t_0^m)}x_0,
\quad
\forall t \geq t_0.
\end{equation}
\end{proposition}

\begin{remark}
Let us suppose that $A_{\alpha}(\cdot)$, $b_\alpha(\cdot)$
verify the relations
{\rm (\ref{ecalfat5.1})} and {\rm (\ref{ecalfat5.2})}.
Let $y \colon \big\{ t \in \mathcal{Z} \, \big|\,
t \geq t_0 \big\} \to K^n$ a particular solution
of the recurrence {\rm (\ref{rec.alfa.lin})}.

Note that for any other solution of the recurrence
{\rm (\ref{rec.alfa.lin})},
$x \colon \big\{ t \in \mathcal{Z} \, \big|\,
t \geq t_0 \big\} \to K^n$,
the function $x(\cdot)-y(\cdot)$ is a solution
of the  recurrence {\rm (\ref{rec.alfa.omog})}.
Hence, according Proposition {\rm \ref{alfa.p8}}, we find
$x(t)-y(t)=\chi(t,t_0) \big( x(t_0)-y(t_0) \big)$,
$\forall t \geq t_0$.

It follows that the unique function
$x \colon \big\{ t \in \mathcal{Z} \, \big|\,
t \geq t_0 \big\} \to K^n$,
which, for any $t \geq t_0$, verifies the recurrence {\rm (\ref{rec.alfa.lin})}
and the initial condition $x(t_0)=x_0$, is
\begin{equation*}
x(t)=\chi(t,t_0) \big( x_0-y(t_0) \big)+y(t),
\quad
\forall t \geq t_0.
\end{equation*}
\end{remark}

\begin{theorem}
\label{alfa.tt7}
We consider the functions
$A_\alpha \colon \big\{ t \in \mathbb{Z}^m \, \big|\,
t \geq t_0 \big\} \to \mathcal{M}_{n}(K)$,
$\alpha \in \{1,2, \ldots, m\}$ (with $t_0 \in \mathbb{Z}^m$),
for which the relations {\rm (\ref{ecalfat5.1})} are
satisfied. We denote
\begin{equation*}
V(t_0)=
\Big \{
x \colon
\big\{ t \in \mathbb{Z}^m \, \big|\,
t \geq t_0 \big\}
\to K^n \, \Big|\,\,
x\,\, \mbox{{\rm solution of the recurrence}}\,
{\rm (\ref{rec.alfa.omog})}
\Big \}.
\end{equation*}
\noindent
$a)$ The set $V(t_0)$ is a $K$ - vector space of dimension $n$.

\noindent
$b)$ Let  $\big\{ v_1, v_2, \ldots, v_n \big \}$ a basis of $K^n$.
For $j\in \{ 1,2, \ldots, n\}$, we consider
$y_j \colon
\big\{ t \in \mathbb{Z}^m \, \big|\,
t \geq t_0 \big\} \to K^n$
as solution of the recurrence {\rm (\ref{rec.alfa.omog})}
which verifies $y_j(t_0)=v_j$.
Then the set
$\big\{ y_1(\cdot), y_2(\cdot), \ldots, y_n(\cdot) \big \}$
is a basis of the vector space $V(t_0)$.

\noindent
$c)$ If $z_1(t)$, $z_2(t)$, \ldots, $z_n(t)$
are the columns of the matrix $\chi(t,t_0)$ (for $t \geq t_0$),
then
$\big\{ z_1(\cdot), z_2(\cdot), \ldots, z_n(\cdot) \big \}$
is a basis of the vector space $V(t_0)$.
\end{theorem}

\begin{proof}
$a)$ and $b)$: One verifies automatically that
$V(t_0)$ is a vector space.

Let $x(\cdot) \in V(t_0)$ and let $x_0=x(t_0)$.
There exist $a_1$, $a_2$, \ldots, $a_n$ $\in K$ such that
\begin{equation*}
x_0=a_1v_1+a_2v_2+\ldots+a_nv_n.
\end{equation*}
Let $y(\cdot)=a_1y_1(\cdot)+a_2y_2(\cdot)+\ldots+a_ny_n(\cdot)$.
Obvious that $y(\cdot)\in V(t_0)$. We find
\begin{equation*}
y(t_0)=a_1y_1(t_0)+a_2y_2(t_0)+\ldots+a_ny_n(t_0)
=
a_1v_1+a_2v_2+\ldots+a_nv_n=x_0
\end{equation*}
Since $x(\cdot)$ and $y(\cdot)$ are solutions of the recurrence {\rm (\ref{rec.alfa.omog})}
and $x(t_0)=y(t_0)=x_0$, by uniqueness property
(Theorem \ref{alfa.t5}), it follows that
$x(\cdot)=y(\cdot)$; consequently
\begin{equation*}
x(\cdot)=a_1y_1(\cdot)+a_2y_2(\cdot)+\ldots+a_ny_n(\cdot).
\end{equation*}
We have proved that $\big\{ y_1(\cdot), y_2(\cdot), \ldots, y_n(\cdot) \big \}$
is a system of generators for the vector space $V(t_0)$.

Let $a_1$, $a_2$, \ldots, $a_n$ $\in K$ such that
$a_1y_1(\cdot)+a_2y_2(\cdot)+\ldots+a_ny_n(\cdot)=0$, i.e.
\begin{equation*}
a_1y_1(t)+a_2y_2(t)+\ldots+a_ny_n(t)=0,
\quad
\forall t\geq t_0.
\end{equation*}

For $t=t_0$, we obtain $\displaystyle \sum_{j=1}^n a_jy_j(t_0)=0$,
i.e., $\displaystyle \sum_{j=1}^n a_jv_j=0$.

Consequently $a_j=0$, $\forall j$.
Hence $ y_1(\cdot), y_2(\cdot), \ldots, y_n(\cdot)$
are linearly independent, i.e.,
$\big\{ y_1(\cdot), y_2(\cdot), \ldots, y_n(\cdot) \big \}$
is a basis of the vector space $V(t_0)$;
the dimension of this space is obviously $n$.

$c)$ Let $\big\{ e_1, e_2, \ldots, e_n \big \}$ be the canonical basis
of the space $K^n$.

Hence $I_n=\Big( e_1 \,\, e_2 \,\, \ldots \,\, e_n  \Big )$.
From the definition of the matrix $\chi(t,t_0)$,
it follows that
$z_j (\cdot)$ is the solution of the recurrence
{\rm (\ref{rec.alfa.omog})}
which verifies $z_j(t_0)=e_j$, $\forall j$.
According the point $b)$, it follows that
$\big\{ z_1(\cdot), z_2(\cdot), \ldots, z_n(\cdot) \big \}$
is a basis of the vector space $V(t_0)$.
\end{proof}

\subsection{Case of non-degenerate matrices} \label{inv}

In this subsection,
we consider the functions
$$
A_\alpha \colon \mathbb{Z}^m \to \mathcal{M}_{n}(K),
\quad
\alpha \in \{1,2, \ldots, m\},
$$
such that
for any
$\alpha\in \{1,2, \ldots, m\}$
and any $t\in \mathbb{Z}^m$,
the matrix $A_{\alpha}(t)$ is invertible
and $\forall t \in \mathbb{Z}^m$,
$\forall \alpha, \beta \in \{1,2, \ldots, m\}$
the relations {\rm (\ref{ecalfat5.1})}
hold.

According the Theorem \ref{alfa.t6},
for any pair
$(t_0,x_0) \in \mathbb{Z}^m \times K^n$,
there exists a unique function
$x \colon \mathbb{Z}^m \to K^n$,
which verifies the recurrence
\begin{equation}
\label{rec.alfa.omog.inv}
x(t+1_\alpha) = A_\alpha(t)x(t),
\quad
\forall t \in \mathbb{Z}^m,
\quad
\forall \alpha \in \{1,2, \ldots, m\},
\end{equation}
and the condition $x(t_0)=x_0$.

In this case the fundamental matrix
can be defined on the set
$\mathbb{Z}^m  \times \mathbb{Z}^m$.
For each $t_0 \in \mathbb{Z}^m$,
$\chi(\,\cdot\, , t_0) \colon \mathbb{Z}^m
\to \mathcal{M}_n(K)$ is the unique matrix solution of the recurrence
\begin{equation*}
X(t+1_\alpha) = A_\alpha(t)X(t),
\quad
\forall \alpha \in \{1,2, \ldots, m\},
\end{equation*}
which verifies $X(t_0)=I_n$. In this way we obtain
the fundamental matrix associated to the recurrence {\rm (\ref{rec.alfa.omog.inv})},
i.e., the function
\begin{equation*}
\chi(\,\cdot\, , \cdot\,) \colon
\mathbb{Z}^m  \times \mathbb{Z}^m \to \mathcal{M}_n(K).
\end{equation*}

The statements in the Proposition \ref{alfa.p7}
are maintained with few changes.

The statement $a)$ rewrites
$\chi(t,s)\chi(s,r)=\chi(t,r)$,\,\,
$\forall t,s,r \in  \mathbb{Z}^m$.
The proof is similar to those given in the proof of Proposition \ref{alfa.p7}.

The point $i)$ becomes
$\chi(t,s)=\chi(t,t_0)\chi(s,t_0)^{-1}$,\,\,
$\forall t,s,t_0 \in  \mathbb{Z}^m$.
For $t_0=t$, we obtain
$\chi(t,s)=\chi(s,t)^{-1}$,\,\,
$\forall t,s \in  \mathbb{Z}^m$.

It may look easy that the point $j)$ can be completed in this way:

{\it if $\forall \alpha \in \{1,2, \ldots, m\}$,
the matrix functions $A_{\alpha}(\cdot)$ are constant, then}
\begin{equation*}
\chi(t,s)=
A_1^{(t^1-s^1)}A_2^{(t^2-s^2)}
\cdot
\ldots
\cdot
A_m^{(t^m-s^m)},
\quad
\forall
t, s \in \mathbb{Z}^m.
\end{equation*}

The analog of the Proposition \ref{alfa.p8} is:

{\it
The solution of the recurrence {\rm (\ref{rec.alfa.omog.inv})}
which verifies $x(t_0)=x_0$, is}
\begin{equation*}
x \colon \mathbb{Z}^m \to K^n,
\quad
x(t)=\chi(t,t_0)x_0,
\quad
\forall t \in \mathbb{Z}^m.
\end{equation*}

{\it
If $\forall \alpha \in \{1,2, \ldots, m\}$,
the matrix functions $A_{\alpha}(\cdot)$ are constants,
then}
\begin{equation*}
x(t)=
A_1^{(t^1-t_0^1)}A_2^{(t^2-t_0^2)}
\cdot
\ldots
\cdot
A_m^{(t^m-t_0^m)}x_0,
\quad
\forall t \in \mathbb{Z}^m.
\end{equation*}

Let $t_0 \in \mathbb{Z}^m$.
We denote
\begin{equation*}
W(t_0)=
\Big \{
x \colon \mathbb{Z}^m \to K^n \, \Big|\,\,
x\,\, \mbox{{\rm solution of the recurrence}}\,
{\rm (\ref{rec.alfa.omog.inv})}
\Big \}.
\end{equation*}

With a proof similar to those for the
Theorem \ref{alfa.tt7}, we obtain the following result.

{\it
\noindent
$a)$ The set $W(t_0)$ is a $K$ - vector space of dimension $n$.

\noindent
$b)$ Let $\big\{ v_1, v_2, \ldots, v_n \big \}$ be a basis of $K^n$.
For $j\in \{ 1,2, \ldots, n\}$, we consider
$y_j \colon \mathbb{Z}^m \to K^n$
as solution of the recurrence {\rm (\ref{rec.alfa.omog.inv})}
which verifies $y_j(t_0)=v_j$.
Then the set
$\big\{ y_1(\cdot), y_2(\cdot), \ldots, y_n(\cdot) \big \}$
is a basis of the space $W(t_0)$.

\noindent
$c)$ If $z_1(t)$, $z_2(t)$, \ldots, $z_n(t)$
are the columns of the matrix fundamental $\chi(t,t_0)$ ($\forall t \in \mathbb{Z}^m$),
then
$\big\{ z_1(\cdot), z_2(\cdot), \ldots, z_n(\cdot) \big \}$
is a basis of the space $W(t_0)$}.

\subsection{Auxiliary result}

\begin{theorem} A solution of the multiple recurrence system
$$x(t+1_\alpha)=A_\alpha(t)x(t), \quad \alpha \in \{1,2, \ldots, m\}$$
is solution of the diagonal recurrence system
$$x(t+{\bf 1}) =A(t)x(t),$$
where
$$A(t) = A_m\left(t+\sum_{\alpha =1}^{m-1}1_\alpha\right)A_{m-1}\left(t+\sum_{\alpha =1}^{m-2}1_\alpha\right)\cdots A_2(t+1_1)A_1(t).$$\end{theorem}

\begin{proof} Taking into account that
$\displaystyle \sum_{\alpha =1}^m 1_\alpha = {\bf 1}$, we can write successively
$$x(t+{\bf 1})= x\left(t+ \sum_{\alpha =1}^m 1_\alpha\right)= x\left(t + \sum_{\alpha =1}^{m-1}1_\alpha +1_m\right)$$
$$=A_m\left(t + \sum_{\alpha =1}^{m-1}1_\alpha \right)x\left(t+\sum_{\alpha =1}^{m-1}1_\alpha \right)$$
$$= A_m\left(t + \sum_{\alpha =1}^{m-1}1_\alpha \right)A_{m-1}\left(t + \sum_{\alpha =1}^{m-2}1_\alpha \right)x\left(t + \sum_{\alpha =1}^{m-2}1_\alpha \right)= ...=A(t)x(t).$$\end{proof}

\section{Solving the linear discrete multitime multiple recurrence
with constant coefficients}

Let $A_1, A_2, \ldots, A_m \in \mathcal{M}_{n}(K)$
be constant matrices such that
$A_{\alpha}A_{\beta}=A_{\beta}A_{\alpha}$,
$\forall \alpha, \beta \in \{1,2, \ldots, m\}$.

Let $t_0\in \mathbb{Z}^m$ and $x_0 \in K^n$.
According the Proposition \ref{alfa.p8},
the function
$x \colon \big\{ t \in \mathbb{Z}^m \, \big|\,
t \geq t_0 \big\} \to K^n$, given by the
formula {\rm (\ref{solautconst})}, is
the solution of the recurrence
\begin{equation}
\label{ecalfaexm2.1}
x(t+1_\alpha) = A_{\alpha} x(t),
\quad
\forall \alpha \in \{1,2, \ldots, m\},
\end{equation}
which verifies $x(t_0)=x_0$.
If $\forall \alpha$, $A_{\alpha}$
is invertible, then according the results in Subsection \ref{inv},
the recurrence {\rm (\ref{ecalfaexm2.1})}
has a unique solution
$x \colon  \mathbb{Z}^m \to K^n$
which verifies $x(t_0)=x_0$.
This is defined by the same formula {\rm (\ref{solautconst})},
but for any $t \in \mathbb{Z}^m$.

We shall use the following result.
\begin{theorem}
\label{diagsimultan}
Let $K$ be a field
and let $\mathcal{F}\neq \emptyset$,
$\mathcal{F}\subseteq \mathcal{M}_n(K)$, such that
any two matrices from $\mathcal{F}$
commute. If any matrix in $\mathcal{F}$
is diagonalizable (over $K$),
then there exists an invertible matrix
$T \in \mathcal{M}_{n}(K)$, such that
$\forall A \in \mathcal{F}$,
$\exists D(A) \in \mathcal{M}_{n}(K)$,
$D(A)$ diagonal matrix, for which
$A=T  D(A)   T^{-1}$.
\end{theorem}

We shall denote by $diag (d_1 ; d_2; \ldots ; d_n) \in \mathcal{M}_{n}(K)$,
the diagonal matrix, which has on the principal diagonal the elements
$d_1,d_2,\ldots, d_n$, in this order.

If the matrix $A \in \mathcal{M}_{n}(K)$
has the columns $q_1, q_2,\ldots, q_n$, we shall denote
$$ A = col(q_1; q_2;\ldots ; q_n). $$

\begin{theorem}
\label{alfa.tt8}
Let $A_1, A_2, \ldots, A_m \in \mathcal{M}_{n}(K)$
be diagonalizable matrices (over $K$), such that
$A_{\alpha}A_{\beta}=A_{\beta}A_{\alpha}$,
$\forall \alpha, \beta \in \{1,2, \ldots, m\}$.

Let $T = col(v_1 ; v_2 ; \ldots ; v_n) \in \mathcal{M}_{n}(K)$
be an invertible matrix such that
\begin{equation*}
A_{\alpha}=T
\cdot
diag (\lambda_{1,\alpha}; \lambda_{2,\alpha};\ldots ; \lambda_{n,\alpha})
\cdot
T^{-1},
\quad
\forall \alpha \in \{1,2, \ldots, m\},
\end{equation*}
where $\lambda_{1,\alpha}, \lambda_{2,\alpha}, \ldots , \lambda_{n,\alpha} \in K$
(such $T$ exists, according Theorem {\rm \ref{diagsimultan}}).

Then, $\forall (t_0,x_0) \in \mathbb{Z}^m \times K^n$,
the solution of the recurrence {\rm (\ref{ecalfaexm2.1})},
which verifies $x(t_0)=x_0$, is
\begin{equation*}
x \colon \big\{ t \in \mathbb{Z}^m \, \big|\,
t \geq t_0 \big\} \to K^n,
\end{equation*}
\begin{equation}
\label{eq.diagsimultan1}
x(t)=\sum_{j=1}^n c_j
\Big (
\prod_{\alpha =1}^m \lambda_{j,\alpha}^{t^{\alpha}-t_0^{\alpha}}
\Big )
v_j,
\quad
\forall t \geq t_0,
\end{equation}
where $(c_1, c_2, \ldots , c_n)^{\top}=T^{-1}x_0$.

If $\forall \alpha$, the matrix $A_{\alpha}$
is invertible, then
the recurrence {\rm (\ref{ecalfaexm2.1})}
has a unique solution
$x \colon  \mathbb{Z}^m \to K^n$
which verifies $x(t_0)=x_0$.
This solution is defined also by the formula
{\rm (\ref{eq.diagsimultan1})},
but for any $t \in \mathbb{Z}^m$.
\end{theorem}

\begin{proof}
According the formula {\rm (\ref{solautconst})},
for any $t \geq t_0$, we have
\begin{equation*}
x(t)=\Big (
\prod_{\alpha =1}^m A_{\alpha}^{t^{\alpha}-t_0^{\alpha}}
\Big )x_0
\end{equation*}
\begin{equation*}
=\Big (
\prod_{\alpha =1}^m
\big(
T
\cdot
diag (\lambda_{1,\alpha}; \lambda_{2,\alpha};\ldots ; \lambda_{n,\alpha})
\cdot
T^{-1}
\big)^{t^{\alpha}-t_0^{\alpha}}
\Big )x_0
\end{equation*}
\begin{equation*}
=\Big (
\prod_{\alpha =1}^m
T
\cdot
\big(
diag (\lambda_{1,\alpha}; \lambda_{2,\alpha};\ldots ; \lambda_{n,\alpha})
\big)^{t^{\alpha}-t_0^{\alpha}}
\cdot
T^{-1}
\Big )x_0
\end{equation*}
\begin{equation*}
=T
\cdot
\Big (
\prod_{\alpha =1}^m
diag \big(
\lambda_{1,\alpha}^{t^{\alpha}-t_0^{\alpha}};
\lambda_{2,\alpha}^{t^{\alpha}-t_0^{\alpha}};
\ldots ; \lambda_{n,\alpha}^{t^{\alpha}-t_0^{\alpha}}
\big)
\Big )\cdot
T^{-1}x_0
\end{equation*}
\begin{equation*}
=
col(v_1 ; v_2 ; \ldots ; v_n)
\cdot
diag \Big(
\prod_{\alpha =1}^m \lambda_{1,\alpha}^{t^{\alpha}-t_0^{\alpha}};
\prod_{\alpha =1}^m \lambda_{2,\alpha}^{t^{\alpha}-t_0^{\alpha}};
\ldots ;
\prod_{\alpha =1}^m \lambda_{n,\alpha}^{t^{\alpha}-t_0^{\alpha}}
\Big)
\cdot
T^{-1}x_0
\end{equation*}
\begin{equation*}
=
col \Bigg(
\Big(\prod_{\alpha =1}^m \lambda_{1,\alpha}^{t^{\alpha}-t_0^{\alpha}}\Big) v_1;
\Big(\prod_{\alpha =1}^m \lambda_{2,\alpha}^{t^{\alpha}-t_0^{\alpha}}\Big) v_2;
\ldots ;
\Big(\prod_{\alpha =1}^m \lambda_{n,\alpha}^{t^{\alpha}-t_0^{\alpha}}\Big)v_n
\Bigg)
\cdot
(c_1, c_2, \ldots , c_n)^{\top}
\end{equation*}
\begin{equation*}
=
\sum_{j=1}^n
c_j
\Big(
\prod_{\alpha =1}^m
\lambda_{j,\alpha}^{t^{\alpha}-t_0^{\alpha}}
\Big) v_j.
\end{equation*}

If all the matrices $A_{\alpha}$ are invertible, then we saw that the formula
{\rm (\ref{solautconst})} is true for any $t \in \mathbb{Z}^m$.
But $A_{\alpha}$ is invertible
iff $\lambda_{j,\alpha}\neq 0$,
$\forall j=\overline{1,n}$.
We notice easily that in this case all equalities
above are true for any $t \in \mathbb{Z}^m$.
\end{proof}

\begin{remark}
\label{obs.alfa.tt8}
If $T = col(v_1 ; v_2 ; \ldots ; v_n) \in \mathcal{M}_{n}(K)$
is the invertible matrix which appears in Theorem {\rm \ref{alfa.tt8}},
then $\big \{ v_1, v_2, \ldots , v_n \big \}$ is a basis
of $K^n$, and each $v_j$
is an eigenvector for all the matrices $A_{\alpha}$.
\end{remark}

\begin{remark}
\label{oobs.alfa.tt8}
In the conditions of Theorem {\rm \ref{alfa.tt8}},
the fundamental matrix is
\begin{equation*}
\chi(t,t_0)
=
T
\cdot
\Big (
\prod_{\alpha =1}^m
diag \big(
\lambda_{1,\alpha}^{t^{\alpha}-t_0^{\alpha}};
\lambda_{2,\alpha}^{t^{\alpha}-t_0^{\alpha}};
\ldots ; \lambda_{n,\alpha}^{t^{\alpha}-t_0^{\alpha}}
\big)
\Big )\cdot
T^{-1},
\quad
\forall t \geq t_0.
\end{equation*}

If all the matrices $A_{\alpha}$ are invertible, then the foregoing formula is true for any
$(t, t_0) \in \mathbb{Z}^m \times \mathbb{Z}^m$.
\end{remark}

For $A \in \mathcal{M}_{n}(K)$ and
$k \in \mathbb{N}$, we denote
\begin{equation*}
S(k;A)
=
\left \{
\begin{array}{ccc}
\!\! I_n+A+\ldots +A^{k-1} \vspace{0.1 cm} & \hbox{if} & k \geq 1\\
\!\! O_n & \hbox{if} & k=0.
\end{array}
\right.
\end{equation*}

\begin{theorem}
\label{alfa.tt9}
For $\alpha \in \{ 1,2, \ldots, m \}$,
we consider the matrices
$A_\alpha \in \mathcal{M}_{n}(K)$,
$b_{\alpha} \in K^n=\mathcal{M}_{n,1}(K)$
such that
\begin{equation}
\label{alfatt9.1}
A_\alpha A_\beta
=
A_\beta A_\alpha,
\quad
\forall \alpha, \beta \in \{ 1,2, \ldots, m \}
\end{equation}
\begin{equation}
\label{alfatt9.2}
(I_n-A_\alpha)b_\beta
=
(I_n-A_\beta)b_\alpha,
\quad
\forall \alpha, \beta \in \{ 1,2, \ldots, m \}.
\end{equation}

Let $(t_0,x_0) \in \mathbb{Z}^m \times K^n$.
The solution of the recurrence
\begin{equation}
\label{alfatt9.3}
x(t+1_\alpha) = A_{\alpha} x(t)+b_{\alpha},
\quad
\forall \alpha \in \{1,2, \ldots, m\},
\end{equation}
which verifies $x(t_0)=x_0$,
is the function
$
x \colon \big\{ t \in \mathbb{Z}^m \, \big|\,
t \geq t_0 \big\} \to K^n
$, defined for any $t \geq t_0$ by
\begin{equation*}
x(t)=\Big (
\prod_{\alpha =1}^m A_{\alpha}^{t^{\alpha}-t_0^{\alpha}}
\Big )x_0
+
S(t^1-t_0^1;A_1)b_1
+
\sum_{\beta=2}^m
\Big (
\prod_{\alpha =1}^{\beta -1} A_{\alpha}^{t^{\alpha}-t_0^{\alpha}}
\Big )
S(t^{\beta}-t_0^{\beta} ; A_{\beta})b_{\beta},
\end{equation*}
if $m \geq 2$, respectively
\begin{equation*}
x(t)=A_1^{t^1-t_0^1}x_0
+
S(t^1-t_0^1;A_1)b_1,
\end{equation*}
if $m =1$.
\end{theorem}

\begin{proof}
According the Theorem~\ref{alfa.t5} and the Remark~\ref{compatconst}
it follows that the recurrence {\rm (\ref{alfatt9.3})}
has a unique solution which verifies $x(t_0)=x_0$.

We prove the statement by induction on $m$, the number of components of $t$.

For $m=1$, one verifies immediately, by direct computations,
that for any $t \geq t_0$, the function $x(t)$ verifies the recurrence {\rm (\ref{alfatt9.3})}
and the condition $x(t_0)=x_0$.

Let $m\geq 2$.
Suppose the statement is true for
$m-1$ and we shall prove it for $m$.
We denote $\tilde{t}=(t^2, \ldots, t^m)$;
$\tilde{t}_0=(t_0^2, \ldots, t_0^m)$.

Let $\tilde{x}(\tilde{t})
=x(t_0^1,\tilde{t})
=x(t_0^1,t^2, \ldots, t^m)$.
If $t^1>t_0^1$, then
\begin{equation*}
x(t)=x(t^1,\tilde{t})=
A_1 x(t^1-1,\tilde{t})+b_1
=A_1^2 x(t^1-2,\tilde{t})+A_1b_1+b_1=
\end{equation*}
\begin{equation*}
=
\ldots
=A_1^k x(t^1-k,\tilde{t})+A_1^{k-1}b_1+\ldots+A_1b_1+b_1
\end{equation*}
\begin{equation*}
=
\ldots
=
A_1^{t^1-t_0^1} x(t_0^1,\tilde{t})
+A_1^{t^1-t_0^1-1}b_1+\ldots+A_1b_1+b_1
\end{equation*}
\begin{equation*}
=
A_1^{t^1-t_0^1} \tilde{x}(\tilde{t})
+
S(t^1-t_0^1;A_1)b_1.
\end{equation*}
We have proved that if $t^1>t_0^1$, then
$x(t)=A_1^{t^1-t_0^1} \tilde{x}(\tilde{t})
+
S(t^1-t_0^1;A_1)b_1$;
relation which is verified immediately for $t^1=t_0^1$.

For $\alpha \in \{ 2, \ldots, m \}$, we denote
$\tilde{1}_{\alpha}=(0,\ldots,0,1,0,\ldots, 0)\in \mathbb{Z}^{m-1}$;
hence $1_{\alpha}=(0,\tilde{1}_{\alpha})$.
For $\alpha \geq 2$ and $t^1=t_0^1$, the relations
{\rm (\ref{alfatt9.3})} become:

$x((t_0^1,\tilde{t})+(0,\tilde{1}_{\alpha}))
= A_\alpha  x(t_0^1,\tilde{t}) +b_\alpha$,\, i.e.,
\begin{equation*}
\tilde{x}(\tilde{t}+\tilde{1}_{\alpha})
= A_\alpha \tilde{x}(\tilde{t}) + b_\alpha,
\quad
\forall\, \tilde{t} \geq \tilde{t}_0,\,\,
\forall \alpha \in \{2, \ldots, m\}.
\end{equation*}
Obviously $\tilde{x}(\tilde{t}_0)
=x(t_0^1,\tilde{t}_0)=x(t_0)=x_0$.
Since $\tilde{t}$ has $m-1$ components,
from the induction hypothesis follows that
$\forall \tilde{t} \geq \tilde{t}_0$ we have
\begin{equation*}
\tilde{x}(\tilde{t})=
\Big (
\prod_{\alpha =2}^m A_{\alpha}^{t^{\alpha}-t_0^{\alpha}}
\Big )x_0
+
S(t^2-t_0^2;A_2)b_2
+
\sum_{\beta=3}^m
\Big (
\prod_{\alpha =2}^{\beta -1} A_{\alpha}^{t^{\alpha}-t_0^{\alpha}}
\Big )
S(t^{\beta}-t_0^{\beta} ; A_{\beta})b_{\beta},
\end{equation*}
if $m \geq 3$, respectively
\begin{equation*}
\tilde{x}(\tilde{t})
=
A_2^{t^2-t_0^2}x_0
+
S(t^2-t_0^2;A_2)b_2,
\end{equation*}
if $m =2$.

Hence, if $m \geq 3$, for any $t\geq t_0$, we have
\begin{equation*}
x(t)=A_1^{t^1-t_0^1} \tilde{x}(\tilde{t})
+
S(t^1-t_0^1;A_1)b_1
\end{equation*}
\begin{equation*}
=
A_1^{t^1-t_0^1}\Big (
\prod_{\alpha =2}^m A_{\alpha}^{t^{\alpha}-t_0^{\alpha}}
\Big )x_0
+
A_1^{t^1-t_0^1}S(t^2-t_0^2;A_2)b_2
+
\end{equation*}
\begin{equation*}
+
\sum_{\beta=3}^m
A_1^{t^1-t_0^1}
\Big (
\prod_{\alpha =2}^{\beta -1} A_{\alpha}^{t^{\alpha}-t_0^{\alpha}}
\Big )S(t^{\beta}-t_0^{\beta} ; A_{\beta})b_{\beta}
+
S(t^1-t_0^1;A_1)b_1
\end{equation*}
\begin{equation*}
=
\Big (
\prod_{\alpha =1}^m A_{\alpha}^{t^{\alpha}-t_0^{\alpha}}
\Big )x_0
+
S(t^1-t_0^1;A_1)b_1
+
\sum_{\beta=2}^m
\Big (
\prod_{\alpha =1}^{\beta -1} A_{\alpha}^{t^{\alpha}-t_0^{\alpha}}
\Big )S(t^{\beta}-t_0^{\beta} ; A_{\beta})b_{\beta}.
\end{equation*}

If $m =2$, for any $t\geq t_0$, we have
\begin{equation*}
x(t)=A_1^{t^1-t_0^1} \tilde{x}(\tilde{t})
+
S(t^1-t_0^1;A_1)b_1
\end{equation*}
\begin{equation*}
=
A_1^{t^1-t_0^1}A_2^{t^2-t_0^2}x_0
+
A_1^{t^1-t_0^1}S(t^2-t_0^2;A_2)b_2
+
S(t^1-t_0^1;A_1)b_1.
\end{equation*}
\end{proof}

\begin{theorem}
\label{alfa.tt10}
Consider the matrices
$A_\alpha \in \mathcal{M}_{n}(K)$,
$b_{\alpha} \in K^n=\mathcal{M}_{n,1}(K)$,
which for any $\alpha, \beta \in \{ 1,2, \ldots, m \}$
satisfy the conditions {\rm (\ref{alfatt9.1})} and {\rm (\ref{alfatt9.2})}.

$a)$
Suppose there exists an index
$\alpha_0 \in \{ 1,2, \ldots, m \}$, for which the matrix
$I_n-A_{\alpha_0}$ is invertible.
Let $v\in K^n$, such that
$(I_n-A_{\alpha_0})v=b_{\alpha_0}$, i.e.
$v=(I_n-A_{\alpha_0})^{-1}b_{\alpha_0}$.
Then
\begin{equation}
\label{alfatt10.1}
(I_n-A_\alpha)v=b_\alpha,
\quad
\forall \alpha \in \{ 1,2, \ldots, m \}.
\end{equation}

$b)$ Suppose there exists $v\in K^n$,
such that, for any
$\alpha \in \{ 1,2, \ldots, m \}$, the relations
{\rm (\ref{alfatt10.1})} are true.

Let $(t_0,x_0) \in \mathbb{Z}^m \times K^n$.
Then, the solution of the recurrence {\rm (\ref{alfatt9.3})},
which verifies $x(t_0)=x_0$,
is the function
$
x \colon \big\{ t \in \mathbb{Z}^m \, \big|\,
t \geq t_0 \big\} \to K^n
$, defined for any $t \geq t_0$ by
\begin{equation}
\label{alfatt10.2}
x(t)
=
\Big (
\prod_{\alpha =1}^m A_{\alpha}^{t^{\alpha}-t_0^{\alpha}}
\Big )
\cdot
(x_0-v)+v.
\end{equation}

If furthermore, $\forall \alpha$, the matrix $A_{\alpha}$
is invertible, then the
recurrence {\rm (\ref{alfatt9.3})}
has a unique solution
$x \colon  \mathbb{Z}^m \to K^n$,
which verifies $x(t_0)=x_0$.
This is defined also by the formula
{\rm (\ref{alfatt10.2})},
but for any $t \in \mathbb{Z}^m$.
\end{theorem}

\begin{proof}
$a)$ From the relation $(I_n-A_\alpha)b_{\alpha_0}
=
(I_n-A_{\alpha_0})b_{\alpha}$
it follows the equality:

\noindent
$(I_n-A_\alpha)(I_n-A_{\alpha_0})v
=
(I_n-A_{\alpha_0})b_{\alpha}$
$\Longleftrightarrow$
$(I_n-A_{\alpha_0})(I_n-A_\alpha)v
=
(I_n-A_{\alpha_0})b_{\alpha}$.
Since $I_n-A_{\alpha_0}$ is invertible,
we obtain $(I_n-A_\alpha)v
=b_{\alpha}$.

$b)$ Let $x(\cdot)$ be the solution of the recurrence {\rm (\ref{alfatt9.3})},
which verifies $x(t_0)=x_0$.
We denote $y(\cdot)=x(\cdot)-v$, i.e. $x(\cdot)=y(\cdot)+v$.
We have
\begin{equation*}
y(t+1_\alpha)+v = A_{\alpha} \big( y(t)+v \big)+b_{\alpha},
\quad
\forall \alpha \in \{1,2, \ldots, m\},
\end{equation*}
\begin{equation*}
\Longleftrightarrow
y(t+1_\alpha) = A_{\alpha} y(t)-(I_n-A_\alpha)v+b_{\alpha},
\quad
\forall \alpha \in \{1,2, \ldots, m\}.
\end{equation*}
Since the relations
{\rm (\ref{alfatt10.1})} are true,
it follows that $y(\cdot)$ is
the solution of the recurrence {\rm (\ref{ecalfaexm2.1})}
which verifies $y(t_0)=x_0-v$.
According the Proposition \ref{alfa.p8},
$\forall t\geq t_0$ we have
\begin{equation}
\label{alfatt10.3}
y(t)
=
\Big (
\prod_{\alpha =1}^m A_{\alpha}^{t^{\alpha}-t_0^{\alpha}}
\Big )
\cdot
(x_0-v).
\end{equation}
From the equality $y(t)=x(t)-v$,
we obtain the relation {\rm (\ref{alfatt10.2})}.

If $\forall \alpha$, the matrix $A_{\alpha}$
is invertible, then according the remarks in the Subsection \ref{inv},
it follows that the relation {\rm (\ref{alfatt10.3})}
is true for any $t \in \mathbb{Z}^m$;
consequently, also the formula {\rm (\ref{alfatt10.2})}
is valid for any $t \in \mathbb{Z}^m$.
\end{proof}

\begin{remark}
The statement:
``there exists $v\in K^n$,
such that, for any
$\alpha$, the relations
{\rm (\ref{alfatt10.1})} are true"
is equivalent to the fact that the recurrence {\rm (\ref{alfatt9.3})} has
a constant solution $x(\cdot)=v$.
\end{remark}

\section{Recurrences on a monoid}

Let $M$ be a nonvoid set, let
$\big ( N, \cdot, E \big )$ be a monoid
and let
$\varphi  \colon N \times M \to M$
be an action of the monoid $N$ on the set $M$, i.e.
\begin{equation}
\label{act}
\varphi (AB, x)= \varphi \big( A, (B, x) \big),\,\,\,
\varphi (e, x)=x,
\quad
\forall A,B \in N,
\forall x \in M.
\end{equation}

For any $A\in N$, $x\in M$, we denote
$\varphi (A, x)=Ax$ (not to be confused with the operation of monoid $N$).
The relations {\rm (\ref{act})} become
\begin{equation*}
(AB) x=  A(B x) ,\,\,\,
ex=x,
\quad
\forall A,B \in N,
\forall x \in M.
\end{equation*}

We denote by ${\cal Z}$ one of the sets
$\mathbb{Z}^m$ or
$ \big\{ t \in \mathbb{Z}^m \, \big|\,
t \geq t_1 \big\}$  (with $t_1 \in \mathbb{Z}^m$).

For each $\alpha \in \{1,2, \ldots, m\}$,
we consider the functions
$A_\alpha \colon \mathcal{Z} \to N$,
which define the recurrence
\begin{equation}
\label{ecalfaexm1.1}
x(t+1_\alpha) = A_\alpha(t)x(t),
\quad
\forall \alpha \in \{1,2, \ldots, m\},
\end{equation}
with the unknown function
$x \colon \big\{ t \in \mathcal{Z} \, \big|\,
t \geq t_0 \big\} \to M$,
$t_0 \in \mathcal{Z}$.

\begin{remark}
\label{alfa.o9}
For
$\big ( N, \cdot, E \big )=\big ( \mathcal{M}_n(K), \cdot, I_n \big )$,
$M=K^n=\mathcal{M}_{n,1}(K)$ and the action
\begin{equation*}
\varphi  \colon \mathcal{M}_n(K) \times K^n \to K^n,
\quad
\varphi (A,x)=Ax,
\quad
\forall A \in \mathcal{M}_n(K) , \,\,
\forall x \in K^n,
\end{equation*}
the recurrence {\rm (\ref{ecalfaexm1.1})} becomes
the recurrence {\rm (\ref{rec.alfa.omog})}.
\end{remark}

With a similar proof with those in the
Theorem \ref{alfa.t5}, it follows
\begin{theorem}
\label{alfa.tt11}
$a)$ If, for any $(t_0,x_0) \in \mathcal{Z} \times M$,
there exists at least one function
$x \colon \big\{ t \in \mathcal{Z} \, \big|\,
t \geq t_0 \big\} \to M$, which,
for any $t\geq t_0$, verifies the recurrence
{\rm (\ref{ecalfaexm1.1})} and the condition $x(t_0)=x_0$, then
\begin{equation}
\label{alfatt11.1}
A_\alpha(t+1_\beta)A_\beta (t) x
=
A_\beta(t+1_\alpha)A_\alpha(t) x
\end{equation}
$$
\forall t \in  \mathcal{Z},\,\,
\forall x \in M,
\quad
\forall \alpha, \beta \in \{1,2, \ldots, m\}.
$$

$b)$ If the relations {\rm (\ref{alfatt11.1})},
are satisfied,
then, for any $(t_0,x_0) \in \mathcal{Z} \times M$,
there exists a unique function $x \colon \big\{ t \in \mathcal{Z} \, \big|\,
t \geq t_0 \big\} \to M$, which,
for any $t\geq t_0$ verifies the recurrence
{\rm (\ref{ecalfaexm1.1})} and the condition $x(t_0)=x_0$.
\end{theorem}

Now we consider the action of the monoid $N$ on himself, $\xi  \colon N \times N \to N$,
defined by
\begin{equation*}
\xi(A, X)= A \cdot X,
\quad
\forall A,X \in N
\quad
(``\cdot"\,
\mbox{is the operation considered on}\, N).
\end{equation*}
In this case, being given the functions
$A_\alpha \colon \mathcal{Z} \to N$,
$\alpha \in \{1,2, \ldots, m\}$,
the analogue of the recurrence {\rm (\ref{ecalfaexm1.1})} is
\begin{equation}
\label{ecmonoid}
X(t+1_\alpha) = A_\alpha(t)X(t),
\quad
\forall \alpha \in \{1,2, \ldots, m\},
\end{equation}
with the unknown function
$X \colon \big\{ t \in \mathcal{Z} \, \big|\,
t \geq t_0 \big\} \to N$,
$t_0 \in \mathcal{Z}$.

By doing like in the proof of the
Theorem \ref{alfa.t5}, it is shown that
\begin{theorem}
\label{alfa.tt12}
$a)$ If, for any $t_0 \in \mathcal{Z}$,
there exists at least one function
$X \colon \big\{ t \in \mathcal{Z} \, \big|\,
t \geq t_0 \big\} \to N$, which,
for any $t\geq t_0$, verifies the recurrence
{\rm (\ref{ecmonoid})} and the condition $X(t_0)=E$, then
\begin{equation}
\label{alfatt12.1}
A_\alpha(t+1_\beta)A_\beta (t)
=
A_\beta(t+1_\alpha)A_\alpha(t)
\end{equation}
$$
\forall t \in  \mathcal{Z},
\quad
\forall \alpha, \beta \in \{1,2, \ldots, m\}.
$$

$b)$ If the relations {\rm (\ref{alfatt12.1})},
are satisfied,
then, for any $(t_0,X_0) \in \mathcal{Z} \times N$,
there exists a unique function $x \colon \big\{ t \in \mathcal{Z} \, \big|\,
t \geq t_0 \big\} \to N$, which,
for any $t\geq t_0$ verifies the recurrence
{\rm (\ref{ecmonoid})} and the condition $X(t_0)=X_0$.
\end{theorem}

\begin{definition}
\label{alfa.d2}
Suppose that the functions
$A_\alpha \colon \mathcal{Z} \to N$,
$\alpha \in \{1,2, \ldots, m\}$,
verify the relations
{\rm (\ref{alfatt12.1})}.

For each $t_0 \in \mathcal{Z}$, we denote
$\chi(\,\cdot\, , t_0) \colon
\big\{ t \in \mathcal{Z} \, \big|\, t \geq t_0 \big\}
\to N$ the unique solution of the recurrence
{\rm (\ref{ecmonoid})} which verifies $X(t_0)=E$.
The function
$$
\chi(\,\cdot\, , \cdot\,)
\colon \big\{ (t,s) \in \mathcal{Z} \times \mathcal{Z}\, \big|\,
t \geq s \big\} \to N
$$
is called the transition (fundamental) function
associated to the recurrence
{\rm (\ref{ecalfaexm1.1})}.
\end{definition}
This is the analog of the fundamental solution associated to the recurrence {\rm (\ref{rec.alfa.omog})},
introduced in the Definition \ref{alfa.d1}.

For $\alpha \in \{1,2, \ldots, m\}$
and $k \in \mathbb{N}$, we consider the function
$C_{\alpha,\, k} \colon \mathcal{Z} \to N$,
defined formally by the relation
{\rm (\ref{Cdef})},
replacing $I_n$ with $E$, but now
$A_\alpha \colon \mathcal{Z} \to N$,
hence
$A_\alpha(\cdot)$
are not matrix functions.

\begin{remark}
\label{alfa.obs6}
If the functions
$A_\alpha \colon \mathcal{Z} \to N$,
$\alpha \in \{1,2, \ldots, m\}$,
verify the relations
{\rm (\ref{alfatt12.1})},
then the Proposition {\rm \ref{alfa.p7}}
write identically, but for
$A_\alpha(\cdot)$, $C_{\alpha,\, k}(\cdot)$,
$\chi(\,\cdot\, , \cdot\,)$ considerate
in this Section (no longer matrices);
the proof is identically to those in the Proposition {\rm \ref{alfa.p7}}.
\end{remark}

Analogous to Proposition {\rm \ref{alfa.p8}}, we have
\begin{proposition}
\label{alfa.pp9}
We consider the functions $A_\alpha \colon \mathcal{Z} \to N$,
$\alpha \in \{1,2, \ldots, m\}$,
for which the relations {\rm (\ref{alfatt12.1})} are
satisfied. Let $(t_0,x_0) \in \mathcal{Z} \times M$.
Then, the unique function
$x \colon \big\{ t \in \mathcal{Z} \, \big|\,
t \geq t_0 \big\} \to M$,
which, for any $t \geq t_0$,
verifies the recurrence {\rm (\ref{ecalfaexm1.1})}
and the condition $x(t_0)=x_0$, is
\begin{equation*}
x(t)=\chi(t,t_0)x_0,
\quad
\forall t \geq t_0.
\end{equation*}

If $\forall \alpha \in \{1,2, \ldots, m\}$,
the functions $A_{\alpha}(\cdot)$ are constant,
then
\begin{equation*}
x(t)=
A_1^{(t^1-t_0^1)}A_2^{(t^2-t_0^2)}
\cdot
\ldots
\cdot
A_m^{(t^m-t_0^m)}x_0,
\quad
\forall t \geq t_0.
\end{equation*}
\end{proposition}

\vspace{0.3 cm}
We return to the recurrence {\rm (\ref{rec.alfa.lin})}, i.e.
\begin{equation*}
x(t+1_\alpha) = A_\alpha(t)x(t)+b_\alpha(t),
\quad
\forall \alpha \in \{1,2, \ldots, m\},
\end{equation*}
where $K$ is a field and
$A_\alpha \colon \mathcal{Z} \to \mathcal{M}_{n}(K)$,
$b_\alpha \colon \mathcal{Z} \to K^n=\mathcal{M}_{n,1}(K)$
verify the relations
{\rm (\ref{ecalfat5.1})} and {\rm (\ref{ecalfat5.2})}.
Let $x \colon \big\{ t \in \mathbb{Z}^m \, \big|\,
t \geq t_0 \big\} \to K^n$ the function which,
for any $t \geq t_0$, verifies the recurrence
{\rm (\ref{rec.alfa.lin})} and the condition
$x(t_0)=x_0$ ($t_0 \in \mathcal{Z}$, $x_0 \in K^n$).

We shall assume in addition that
$\forall \alpha \in \{1,2, \ldots, m\}$,
$\forall t \in \mathcal{Z}$, and that the matrix $A_\alpha(t)$
is invertible.

Let $\chi(\cdot , \cdot)$
be the transition (fundamental) matrix associated to
the linear homogeneous recurrence
{\rm (\ref{rec.alfa.omog})}.
According the Proposition {\rm \ref{alfa.p7}},
$\forall t\geq s$, the matrix $\chi(t , s)$ is invertible.
Since $\forall t\geq t_0$,
$\chi(t+  1_\alpha , t_0)=A_{\alpha}(t)  \chi(t,t_0)$
or $A_{\alpha}(t) =\chi(t+  1_\alpha , t_0)\chi(t,t_0)^{-1}$,
it follows that the equality {\rm (\ref{rec.alfa.lin})}
is equivalent to
\begin{equation*}
x(t+1_{\alpha}) =
\chi(t+  1_\alpha , t_0)\chi(t,t_0)^{-1}x(t)+ b_{\alpha}(t)
\end{equation*}
\begin{equation*}
\Longleftrightarrow
\chi(t+  1_\alpha , t_0)^{-1} x(t+1_{\alpha})
=  \chi(t,t_0)^{-1}x(t)+ \chi(t+  1_\alpha , t_0)^{-1}b_{\alpha}(t)
\end{equation*}

Let $\widetilde{x} \colon
\big\{ t \in \mathbb{Z}^m \, \big|\,
t \geq t_0 \big\} \to K^n$,
$\widetilde{x}(t)=\chi(t,t_0)^{-1}x(t)$,
$\forall t \geq t_0$.

For $\alpha \in \{1,2, \ldots, m\}$, let
\begin{equation*}
\widetilde{A}_{\alpha} \colon
\big\{ t \in \mathbb{Z}^m \, \big|\,
t \geq t_0 \big\} \to K^n,
\quad
\widetilde{A}_{\alpha}(t)=
\chi(t+  1_\alpha , t_0)^{-1}b_{\alpha}(t),
\quad
\forall t \geq t_0.
\end{equation*}

We have
$\widetilde{x}(t_0)=\chi(t_0,t_0)^{-1}x(t_0)=x(t_0)$.

From the above it follows that $x(\cdot)$
is a solution of the recurrence {\rm (\ref{rec.alfa.lin})}
which verifies $x(t_0)=x_0$,
if and only if $\widetilde{x}(\cdot)$ is the solution of the recurrence
\begin{equation}
\label{retilda}
\widetilde{x}(t+1_{\alpha}) =
\widetilde{A}_{\alpha}(t)+\widetilde{x}(t),
\quad
\forall t\geq t_0,
\quad
\forall \alpha \in \{1,2, \ldots, m\},
\end{equation}
which verifies $\widetilde{x}(t_0)=x_0$.

We find that the recurrence
{\rm (\ref{retilda})} is of type {\rm (\ref{ecalfaexm1.1})}, where

$\big ( N, \cdot, E \big )=\big ( K^n, + , 0 \big )$,
$M=N=K^n$ and the action is
\begin{equation*}
\psi \colon K^n \times K^n \to K^n,
\quad
\psi (\widetilde{A},\widetilde{x})=\widetilde{A}+\widetilde{x},
\quad
\forall \widetilde{A} \in K^n, \,\,
\forall \widetilde{x} \in K^n.
\end{equation*}

The relation {\rm (\ref{alfatt12.1})}
corresponding to the recurrence
{\rm (\ref{retilda})} is
\begin{equation*}
\widetilde{A}_\alpha(t+1_\beta)+\widetilde{A}_\beta(t)
=
\widetilde{A}_\beta(t+1_\alpha)+\widetilde{A}_\alpha(t)
\end{equation*}
\begin{equation*}
\Longleftrightarrow
\chi(t+1_\beta+  1_\alpha , t_0)^{-1}b_{\alpha}(t+1_\beta)
+
\chi(t+  1_\beta , t_0)^{-1}b_{\beta}(t)
\end{equation*}
\begin{equation*}
=
\chi(t+1_\alpha + 1_\beta , t_0)^{-1}b_{\beta}(t+1_\alpha)
+
\chi(t+  1_\alpha , t_0)^{-1}b_{\alpha}(t)
\end{equation*}
\begin{equation*}
\Longleftrightarrow
b_{\alpha}(t+1_\beta)
+
\chi(t+1_\alpha + 1_\beta , t_0)\chi(t+  1_\beta , t_0)^{-1}b_{\beta}(t)
\end{equation*}
\begin{equation*}
=
b_{\beta}(t+1_\alpha)
+
\chi(t+1_\alpha + 1_\beta , t_0)\chi(t+  1_\alpha , t_0)^{-1}b_{\alpha}(t)
\end{equation*}
\begin{equation*}
b_{\alpha}(t+1_\beta)
+
A_\alpha(t+1_\beta)b_{\beta}(t)
=
b_{\beta}(t+1_\alpha)
+
A_\beta(t+1_\alpha)b_{\alpha}(t)
\end{equation*}
\begin{equation*}
\Longleftrightarrow
A_\alpha(t+1_\beta)b_\beta(t)+b_\alpha(t+1_\beta)
=
A_\beta(t+1_\alpha)b_\alpha(t)+b_\beta(t+1_\alpha),
\end{equation*}
and this is the relation {\rm (\ref{ecalfat5.2})},
which is satisfied.

Let $\widetilde{\chi}(\,\cdot\, , \cdot\,)
\colon \big\{ (t,s) \in \mathbb{Z}^m \times \mathbb{Z}^m\, \big|\,
t \geq s \geq t_0 \big\} \to K^n$
be the fundamental function associated to the
recurrence {\rm (\ref{retilda})}.

According Proposition {\rm \ref{alfa.pp9}}, we have\,
$\widetilde{x}(t)=\widetilde{\chi}(t,t_0)+x_0$.

Since $x(t)=\chi(t,t_0)\widetilde{x}(t)$, it follows that
\begin{equation*}
x(t)=\chi(t,t_0)x_0 + \chi(t,t_0)\widetilde{\chi}(t,t_0).
\end{equation*}

According to the Remark \ref{alfa.obs6}, the matrix
$\widetilde{\chi}(t,t_0)$ writes as a sum of matrices $\widetilde{C}_{\alpha,\, k}(\cdot)$
(analogue of the relation in the step $f)$ of
Proposition \ref{alfa.p7}, but with the operation $``+"$
instead of multiplication), and
\begin{equation*}
\widetilde{C}_{\alpha,\, k}(t)=
\displaystyle \sum_{j=1}^k \widetilde{A}_\alpha(t+(k-j)\cdot 1_\alpha),\,\,
\mbox{if}\,\, k \geq 1,\,\,
\mbox{and}\,\, \widetilde{C}_{\alpha,\, 0}(t)=0,
\end{equation*}
i.e. the analog of the formula (\ref{Cdef}).

\section*{Acknowledgments}

The work has been funded by the Sectoral Operational Programme Human Resources
Development 2007-2013 of the Ministry of European Funds through
the Financial Agreement POSDRU/159/1.5/S/132395.

Partially supported by University Politehnica of Bucharest and by Academy of Romanian Scientists.
Special thanks goes to Prof. Dr. Ionel \c Tevy, who was willing to participate in our discussions
about multivariate sequences and to suggest the name ``multiple recurrences".

\end{document}